\documentclass[12pt]{amsart}
\usepackage{latexsym}
\usepackage{amsfonts}
\usepackage{amsthm}
\usepackage{amsmath}
\usepackage{amssymb}
\usepackage{setspace}
\usepackage{eucal}
\usepackage{tikz}
\usepackage{graphicx}
\usepackage{enumerate}
\usepackage{ytableau}
\usepackage{caption}
\usepackage{subcaption}
\usepackage{hyperref}

\numberwithin{equation}{section}

\newtheorem{theorem}{Theorem}[section]
\newtheorem{lemma}[theorem]{Lemma}

\newtheorem{corollary}[theorem]{Corollary}
\newtheorem{proposition}[theorem]{Proposition}

\theoremstyle{definition}

\newtheorem{remark}{Remark}[section]
\newtheorem{example}[theorem]{Example}
\newtheorem*{claim}{Claim}

\theoremstyle{remark}
\newtheorem{case}{Case}
\newtheorem{subcase}{Case}
\numberwithin{subcase}{case}

\newcommand{\er}{\mathrm{er}}
\newcommand{\pr}{\mathrm{pr}}
\newcommand{\pro}{\mathcal{P}}

\newcommand{\evac}{\mathcal{E}}
\newcommand{\dualevac}{\mathcal{E}^*}
\newcommand{\rot}{\mathtt{rot}}
\newcommand{\flag}{\mathcal{F}}
\newcommand{\skewflag}{\Finv}
\newcommand{\maxv}[1]{\mathrm{max}(#1)}
\newcommand{\maj}{\mathrm{maj}}
\newcommand{\inc}[2]{\mathrm{Inc}_#2(#1)}
\newcommand{\twinc}[2]{\mathrm{Inc}_#2(2 \times #1)}
\newcommand{\syt}[1]{\mathrm{SYT}(#1)}
\newcommand{\hgd}[1]{\mathrm{hgd}(#1)}
\newcommand{\blank}{\phantom{2}}

\renewcommand{\binom}[2]{\genfrac{(}{)}{0pt}{}{#1}{#2}}
\newcommand{\qbinom}[2]{\genfrac{[}{]}{0pt}{}{#1}{#2}_q}

\newcommand{\N}{\mathbb{N}}
\newcommand{\Z}{\mathbb{Z}}

\newcommand{\R}{\mathbb{R}}

\newmuskip\pFqskip
\mathchardef\pFcomma=\mathcode`, 

\newcommand*\pFq[5]{%
  \begingroup
  \begingroup\lccode`~=`,
    \lowercase{\endgroup\def~}{\pFcomma\mkern\pFqskip}%
  \mathcode`,=\string"8000
  {}_{#1}F_{#2}\biggl(\genfrac..{0pt}{}{#3}{#4};#5\biggr)%
  \endgroup
}

\begin{document}

\title{Cyclic Sieving of Increasing Tableaux and small Schr\"oder Paths}
\author{Oliver Pechenik}
\address{Department of Mathematics \\ University of Illinois at Urbana--Champaign \\ Urbana, IL 61801 \\ USA}
\email{pecheni2@illinois.edu}
\date{\today}

\begin{abstract}
An \emph{increasing tableau} is a semistandard tableau with strictly increasing rows and columns. It is well known that the Catalan numbers enumerate both rectangular standard Young tableaux of two rows and also Dyck paths. We generalize this to a bijection between rectangular 2-row increasing tableaux and small Schr\"oder paths.  We demonstrate relations between the jeu de taquin for increasing tableaux developed by H.~Thomas and A.~Yong and the combinatorics of tropical frieze patterns. We then use this jeu de taquin to present new instances of the cyclic sieving phenomenon of V.~Reiner, D.~Stanton, and D.~White, generalizing results of D.~White and of J.~Stembridge.
\end{abstract}

\maketitle 

\section{Introduction}\label{sec:introduction}
An \emph{increasing tableau} is a semistandard tableau such that all rows and columns are strictly increasing and the set of entries is an initial segment of $\Z_{>0}$. For $\lambda$ a partition of $N$, we write $|\lambda|=N$.  We denote by $\inc{\lambda}{k}$ the set of increasing tableaux of shape $\lambda$ with maximum value $|\lambda|- k$. Similarly $\syt{\lambda}$ denotes standard Young tableaux of shape $\lambda$. Notice $\inc{\lambda}{0} = \syt{\lambda}$. We routinely identify a partition $\lambda$ with its Young diagram; hence for us the notations $\syt{m \times n}$ and $\syt{n^m}$ are equivalent.

 A \emph{small Schr\"oder path} is a planar path from the origin to $(n,0)$ that is constructed from three types of line segment: \emph{upsteps} by $(1,1)$, \emph{downsteps} by $(1,-1)$, and \emph{horizontal steps} by $(2,0)$, so that the path never falls below the horizontal axis and no horizontal step lies on the axis. The $n$th \emph{small Schr\"oder number} is defined to be the number of such paths. A \emph{Dyck path} is a small Schr\"oder path without horizontal steps.

Our first result is an extension of the classical fact that Catalan numbers enumerate both Dyck paths and rectangular standard Young tableaux of two rows, $\syt{2 \times n}$. For $T \in \inc{2 \times n}{k}$, let $\maj(T)$ be the sum of all $i$ in row 1 such that $i+1$ appears in row 2. 
\begin{theorem}\label{thm:enumeration}
There are explicit bijections between $\twinc{n}{k}$, small Schr\"oder paths with $k$ horizontal steps, and $\syt{n-k, n-k, 1^k}$. This implies the identity \begin{equation}\label{eq:q-enumerator_in_section_1}\sum_{T \in \twinc{n}{k}} q^{\maj(T)} = q^{n + \binom{k}{2}} \frac{\qbinom{n-1}{k} \qbinom{2n-k}{n-k-1}}{[n-k]_q}.\end{equation} In particular, the total number of increasing tableaux of shape $2 \times n$  is the $n$th small Schr\"oder number.
\end{theorem}
\noindent The ``flag-shaped'' standard Young tableaux of Theorem \ref{thm:enumeration} were previously considered by R.~Stanley \cite{stanley:flag_tableaux} in relation to polygon dissections.

Suppose $X$ is a finite set, ${\mathcal C}_n = \langle c \rangle$ a cyclic group acting on $X$, and $f \in \Z[q]$ a polynomial. The triple $(X,\mathcal{C}_n, f)$ has the  \emph{cyclic sieving phenomenon} \cite{reiner.stanton.white} if for all $m$, the number of elements of $X$ fixed by $c^m$ is $f(\zeta^m)$, where $\zeta$ is any primitive $n$th root of unity. D.~White \cite{white} discovered a cyclic sieving for $2 \times n$ standard Young tableaux. For this, he used a $q$-analogue of the hook-length formula (that is, a $q$-analogue of the Catalan numbers) and a group action by jeu de taquin promotion. B.~Rhoades  \cite[Theorem~1.3]{Rhoades} generalized this result from $\inc{2 \times n}{0}$ to $\inc{m \times n}{0}$. Our main result is a generalization of D.~White's result in another direction, from $\twinc{n}{0}$ to $\twinc{n}{k}$.

We first define \emph{K-promotion} for increasing tableaux.  Define the \emph{SE-neighbors} of a box to be the (at most two) boxes immediately below it or right of it. Let $T$ be an increasing tableau with maximum entry $M$. Delete the entry 1 from $T$, leaving an empty box. Repeatedly perform the following operation simultaneously on all empty boxes until no empty box has a SE-neighbor: Label each empty box by the minimal label of its SE-neighbors and then remove that label from the SE-neighbor(s) in which it appears. If an empty box has no SE-neighbors, it remains unchanged. We illustrate the local changes in Figure \ref{fig:local_moves}. 

\ytableausetup{boxsize=1.1em}

\begin{figure}[h]
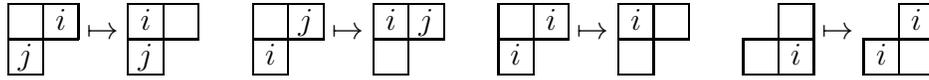

\ytableaushort{\blank i,j} $\; \mapsto$ \ytableaushort{i \blank,j} \hspace{6mm}
\ytableaushort{\blank j,i} $\; \mapsto$ \ytableaushort{ij, \blank} \hspace{6mm}
\ytableaushort{\blank i,i} $\; \mapsto$ \ytableaushort{i \blank, \blank} \hspace{6mm}
\ytableaushort{\none \blank, \blank i} $\; \mapsto$ \ytableaushort{\none i, i \blank}
\caption{Local changes during K-promotion for $i<j$.}\label{fig:local_moves}
\end{figure}

\noindent Notice that the number of empty boxes may change during this process. Finally we obtain the K-promotion $\pro(T)$ by labeling all empty boxes by $M+1$ and then subtracting one from every label. Figure \ref{fig:promotion} shows a full example of K-promotion. 

\begin{figure}[h]
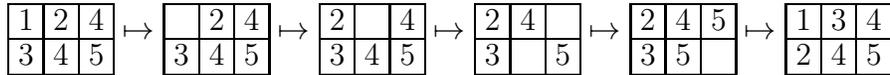

\ytableaushort{124,345} $\; \mapsto$ \ytableaushort{\blank 2 4, 345} $\; \mapsto$ \ytableaushort{2 \blank 4,345} $\; \mapsto$  \ytableaushort{24 \blank,3 \blank 5} $\; \mapsto$  \ytableaushort{245,35 \blank} $\; \mapsto$  \ytableaushort{134,245}
\caption{K-promotion.}\label{fig:promotion}
\end{figure}

Our definition of K-promotion is analogous to that of ordinary promotion, but uses the K-jeu de taquin of H.~Thomas--A.~Yong \cite{thomas.yong:K} in place of ordinary jeu de taquin. (The `K' reflects their original development of K-jeu de taquin in application to K-theoretic Schubert calculus.) Observe that on standard Young tableaux, promotion and K-promotion coincide. 

K-evacuation \cite[$\mathsection 4$]{thomas.yong:K} is defined as follows. Let $T$ be a increasing tableau with maximum entry $M$, and let $[T]_j$ denote the Young diagram consisting of those boxes of $T$ with entry $i \leq j$. Then the K-evacuation $\evac(T)$ is the increasing tableau encoded by the chain in Young's lattice $([\pro^{M - j}(T)]_j)_{0 \leq j \leq M}$. Like ordinary evacuation, $\evac$ is an involution.

Let the non-identity element of $\mathcal{C}_2$ act on $\twinc{n}{k}$ by K-evacuation. We prove the following cyclic sieving, generalizing a result of Stembridge \cite{stembridge:evacuation}.

\begin{theorem}\label{thm:evacuation_CSP}
For all $n$ and $k$, the triple $\big(\twinc{n}{k}, \mathcal{C}_{2}, f\big)$ has the cyclic sieving phenomenon, where \begin{equation}\label{eq:f}f(q):=\frac{\qbinom{n-1}{k} \qbinom{2n-k}{n-k-1}}{[n-k]_q}\end{equation} is the $q$-enumerator from Theorem \ref{thm:enumeration}.
\end{theorem}

We will then need:

\begin{theorem}\label{thm:orbit_size}
For all $n$ and $k$, there is an action of the cyclic group $\mathcal{C}_{2n-k}$ on $T \in \twinc{n}{k}$, where a generator acts by K-promotion.
\end{theorem}

In the case $k=0$, Theorem \ref{thm:orbit_size} is implicit in work of M.-P.~Sch\"{u}tzenberger (\emph{cf.}~\cite{haiman, stanley:promotion}). The bulk of this paper is devoted to proofs of Theorem \ref{thm:orbit_size}, which we believe provide different insights. Finally we construct the following cyclic sieving.

\begin{theorem}\label{thm:promotion_CSP}
For all $n$ and $k$, the triple $\big(\twinc{n}{k}, \mathcal{C}_{2n-k}, f\big)$ has the cyclic sieving phenomenon.
\end{theorem}

Our proof of Theorem~\ref{thm:evacuation_CSP} is by reduction to a result of J.~Stembridge \cite{stembridge:evacuation}, which relies on results about the Kazhdan--Luszig cellular representation of the symmetric group. Similarly, all proofs \cite{Rhoades, purbhoo, fontaine.kamnitzer} of B.~Rhoades' theorem for standard Young tableaux use representation theory or geometry. (Also \cite{webs}, giving new proofs of the 2- and 3-row cases of B.~Rhoades' result, uses representation theory.) In contrast, our proof of Theorem \ref{thm:promotion_CSP} is completely elementary. It is natural to ask also for such representation-theoretic or geometric proofs of Theorem~\ref{thm:promotion_CSP}. We discuss obstacles in Section~\ref{sec:representation_theory}. We do not know a common generalization of our Theorem~\ref{thm:promotion_CSP} and B.~Rhoades' theorem to $\inc{m \times n}{k}$. One obstruction is that for $k>0$, Theorem~\ref{thm:orbit_size} does not generalize in the obvious way to tableaux of more than 3 rows (\emph{cf.}~Example~\ref{ex:counterexample_to_orbit_size_conjecture}).

This paper is organized as follows. In Section \ref{sec:enumerations}, we prove Theorem \ref{thm:enumeration}. We include an additional bijection (to be used in Section~\ref{sec:CSP}) between $\twinc{n}{k}$ and certain noncrossing partitions that we interpret as generalized noncrossing matchings. In Section~\ref{sec:schroeder_paths}, we use the combinatorics of small Schr\"oder paths to prove Theorem \ref{thm:orbit_size} and a characterization of K-evacuation necessary for Theorem~\ref{thm:evacuation_CSP}. We also provide a counterexample to the naive generalization of Theorem~\ref{thm:orbit_size} to 4-row increasing tableaux. In Section~\ref{sec:tropical_frieze}, we make connections with tropicalizations of Conway--Coxeter frieze patterns and demostrate a frieze-diagrammatic approach to some of the key steps in the previous section. In Section \ref{sec:representation_theory}, we prove Theorem~\ref{thm:evacuation_CSP} by interpreting it representation-theoretically in the spirit of \cite{stembridge:evacuation} and \cite{Rhoades}, and discuss obstacles to such an interpretation of Theorem~\ref{thm:promotion_CSP}. In Section \ref{sec:CSP}, we use noncrossing partitions to give another proof of Theorem \ref{thm:orbit_size} and to prove Theorem~\ref{thm:promotion_CSP}.

\section{Bijections and Enumeration}\label{sec:enumerations}

\begin{proposition}\label{prop:flag_bijection}
There is an explicit bijection between $\twinc{n}{k}$ and $\syt{n-k, n-k, 1^k}$.
\end{proposition}
\begin{proof}
Let $T \in \twinc{n}{k}$. The following algorithm produces a corresponding $S \in \syt{n-k, n-k, 1^k}$. Observe that every value in $\{1, \dots, 2n-k\}$ appears in $T$ either once or twice. Let $A$ be the set of numbers that appear twice. Let $B$ be the set of numbers that appear in the second row immediately right of an element of $A$. Note $|A| = |B| = k$.

Let $T'$ be the tableau of shape $(n-k,n-k)$ formed by deleting all elements of $A$ from the first row of $T$ and all elements of $B$ from the second. The standard Young tableau $S$ is given by appending $B$ to the first column. An example is shown in Figure \ref{fig:bijections}.

This algorithm is reversible. Given the standard Young tableau $S$ of shape $(n-k, n-k, 1^k)$, let $B$ be the set of entries below the first two rows. By inserting $B$ into the second row of $S$ while maintaining increasingness, we reconstruct the second row of $T$. Let $A$ be the set of elements immediately left of an element of $B$ in this reconstructed row. By inserting $A$ into the first row of $S$ while maintaining increasingness, we reconstruct the first row of $T$.
\end{proof}

\begin{corollary}\label{cor:q-enumeration}
For all $n$ and $k$ the identity (\ref{eq:q-enumerator_in_section_2}) holds:
\begin{equation}\label{eq:q-enumerator_in_section_2}\sum_{T \in \twinc{n}{k}} q^{\maj(T)} = q^{n + \binom{k}{2}} \frac{\qbinom{n-1}{k} \qbinom{2n-k}{n-k-1}}{[n-k]_q}.\end{equation}
\end{corollary}
\begin{proof}
Observe that $\maj(T)$ for a 2-row rectangular increasing tableau $T$ is the same as the major index of the corresponding standard Young tableau.  The desired $q$-enumerator follows by applying the $q$-hook-length formula to those standard Young tableaux (\emph{cf.}~\cite[Corollary~7.21.5]{stanley:ec}).
\end{proof}

\begin{proof}[Proof of Theorem \ref{thm:enumeration}]
The bijection between $\twinc{n}{k}$ and $\syt{n-k, n-k, 1^k}$ is given by Proposition \ref{prop:flag_bijection}. The $q$-enumeration (\ref{eq:q-enumerator_in_section_1}) is exactly Corollary \ref{cor:q-enumeration}.

We now give a bijection between $\twinc{n}{k}$ and small Schr\"oder paths with $k$ horizontal steps. Let $T \in \twinc{n}{k}$. For each integer $j$ from 1 to $2n-k$, we create one segment of a small Schr\"oder path $P_T$. If $j$ appears only in the first row,  then the $j$th segment of $P_T$ is an upstep. If $j$ appears only in the second row of $T$, the $j$th segment of $P_T$ is a downstep. If $j$ appears in both rows of $T$, the $j$th segment of $P_T$ is horizontal. It is clear that the tableau $T$ can be reconstructed from the small Schr\"oder path $P_T$, so this operation gives a bijection. Thus increasing tableaux of shape $(n,n)$ are counted by small Schr\"oder numbers.

A bijection between small Schr\"oder paths with $k$ horizontal steps and $\syt{n-k, n-k, 1^k}$ may be obtained by composing the two previously described bijections. 
\end{proof}

\begin{figure}[h]
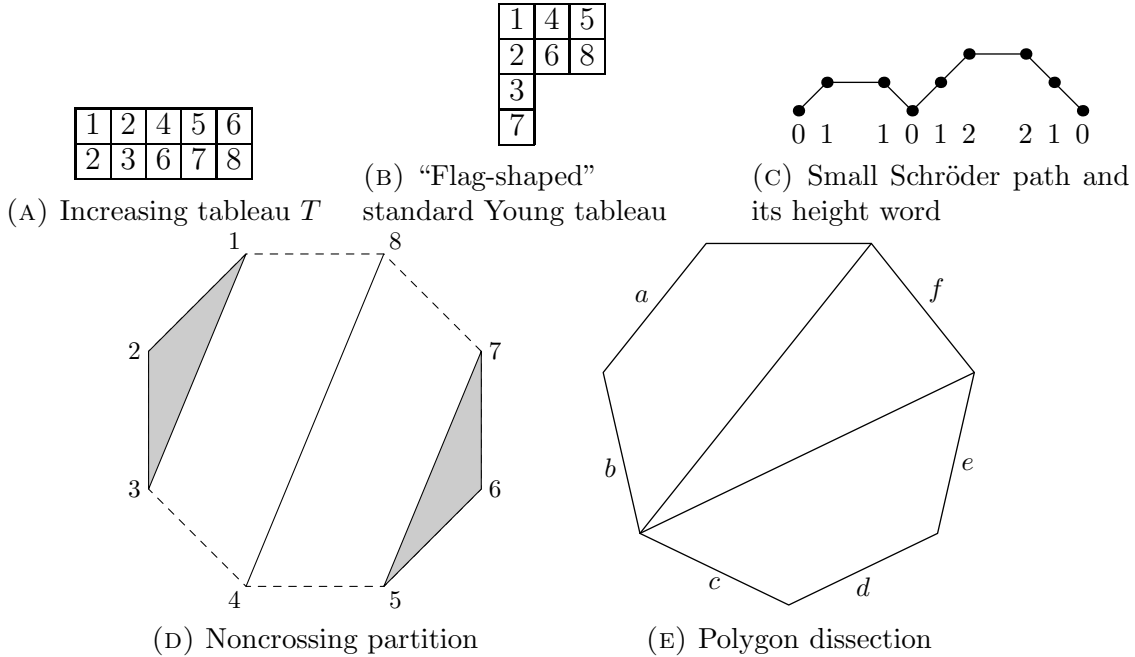

	\begin{subfigure}[b]{0.3\textwidth}
		\centering
		\ytableaushort{12456,23678}
		\caption{Increasing tableau $T$}
	\end{subfigure}
		\begin{subfigure}[b]{0.3\textwidth}
		\centering
		\ytableaushort{145,268,3,7}
		\caption{``Flag-shaped'' \\ standard Young tableau}
	\end{subfigure}
	\begin{subfigure}[b]{0.3\textwidth}
		\centering
		 \includegraphics[width=4cm]{smallschroederpath.mps}
		\caption{Small Schr\"oder path and its height word}
	\end{subfigure}
	\begin{subfigure}[b]{0.3\textwidth}
		\centering
		\includegraphics[width=5cm]{noncrossing.mps}
		\caption{Noncrossing partition}
	\end{subfigure}
	\hspace{1cm}
	\begin{subfigure}[b]{0.3\textwidth}
		\centering
		 \includegraphics[width=5cm]{dissection.mps}
		\caption{Polygon dissection}
	\end{subfigure}
\caption{A rectangular increasing tableau $T \in \inc{5,5}{2}$ with its corresponding standard Young tableau of shape $(3,3,1,1)$, small Schr\"oder path, noncrossing partition of $\{1, \dots, 8\}$ with all blocks of size at least two, and heptagon dissection.}\label{fig:bijections}
\end{figure}

For increasing tableaux of arbitrary shape, there is unlikely to be a product formula like the hook-length formula for standard Young tableaux or our Theorem \ref{thm:enumeration} for the 2-row rectangular case. For example, we compute that $\inc{4,4,4}{2} = 2^2 \cdot 3 \cdot 7 \cdot 19$  and that there are $3 \cdot 1531$ increasing tableaux of shape $(4,4,4)$ in total.

The following bijection will play an important role in our proof of Theorem \ref{thm:promotion_CSP} in Section~\ref{sec:CSP}. A partition of $\{1, \dots, N\}$ is \emph{noncrossing} if the convex hulls of the blocks are pairwise disjoint when the values $1, \dots, N$ are equally spaced around a circle with 1 in the upper left and values increasing counterclockwise (\emph{cf.}~Figure \ref{fig:bijections}(D)).

\begin{proposition}\label{prop:partition_bijection}
There is an explicit bijection between $\twinc{n}{k}$ and noncrossing partitions of $2n-k$ into $n-k$ blocks all of size at least 2.
\end{proposition}
\begin{proof}
Let $T \in \twinc{n}{k}$. For each $i$ in the second row of $T$, let $s_i$ be the largest number in the first row that is less than $i$ and that is not $s_j$ for some $j<i$. Form a partition of $2n-k$ by declaring, for every $i$, that $i$ and $s_i$ are in the same block. We see this partition has $n-k$ blocks by observing that the largest elements of the blocks are precisely the numbers in the second row of $T$ that do not also appear in the first row. Clearly there are no singleton blocks. 

If the partition were \emph{not} noncrossing, there would exist some elements $a < b < c < d$ with $a,c$ in a block $B$ and $b,d$ in a distinct block $B'$. Observe that $b$ must appear in the first row of $T$ and $c$ must appear in the second row of $T$ (not necessarily exclusively). We may assume $c$ to be the least element of $B$ that is greater than $b$. We may then assume $b$ to be the greatest element of $B'$ that is less than $c$. Now consider $s_c$, which must exist since $c$ appears in the second row of $T$. By definition, $s_c$ is the largest number in the first row that is less than $c$ and that is not $s_j$ for some $j<c$. By assumption, $b$ appears in the first row, is less than $c$, and is not $s_j$ for any $j<c$; hence $s_c \geq b$. Since however $b$ and $c$ lie in distinct blocks, $s_c \neq b$, whence $b < s_c < c$. This is impossible, since we took  $c$ to be the least element of $B$ greater than $b$. Thus the partition is necessarily noncrossing.

To reconstruct the increasing tableau, read the partition from 1 to $2n-k$. Place the smallest elements of blocks in only the first row, place the largest elements of blocks in only the second row, and place intermediate elements in both rows.
\end{proof}

The set $\twinc{n}{k}$ is also in bijection with $(n+2)$-gon dissections by $n-k-1$ diagonals. We do not describe this bijection, as it is well known (\emph{cf.}~\cite{stanley:flag_tableaux}) and will not be used except in Section \ref{sec:CSP} for comparison with previous results.  The existence of a connection between increasing tableaux and polygon dissections was first suggested in \cite{thomas.yong:long_sequence}. An example of all these bijections is shown in Figure \ref{fig:bijections}. 

\begin{remark}
A \emph{noncrossing  matching} is a noncrossing partition with all blocks of size two. Like Dyck paths, polygon triangulations, and 2-row rectangular standard Young tableaux, noncrossing matchings are enumerated by the Catalan numbers. Since increasing tableaux were developed as a K-theoretic analogue of standard Young tableaux, it is tempting also to regard small Schr\"oder paths,  polygon dissections, and noncrossing partitions without singletons as K-theory analogues of Dyck paths, polygon triangulations, and noncrossing matchings, respectively. In particular, by analogy with \cite{webs}, it is tempting to think of noncrossing partitions without singletons as ``K-webs'' for $\mathfrak{sl}_2$, although their representation-theoretic significance is unknown.
\end{remark}

\section{K-Promotion and K-Evacuation} \label{sec:schroeder_paths}
In this section, we prove Theorem \ref{thm:orbit_size}, as well as a proposition important for Theorem~\ref{thm:evacuation_CSP}. Let $\maxv{T}$ denote the largest entry in a tableau $T$. For a tableau $T$, we write $\rot(T)$ for the (possibly skew) tableau formed by rotating $180$ degrees and reversing the alphabet, so that label $x$ becomes $\maxv{T} + 1 - x$. Define \emph{dual K-evacuation} $\dualevac$ by $\dualevac := \rot  \circ \evac \circ \rot$. (This definition of $\dualevac$ strictly makes sense only for rectangular tableaux. For a tableau $T$ of general shape $\lambda$, in place of applying $\rot$, one should dualize $\lambda$ (thought of as a poset) and reverse the alphabet. We will not make any essential use of this more general definition.)

Towards Theorem \ref{thm:orbit_size}, we first prove basic combinatorics of the above operators that are well-known in the standard Young tableau case (\emph{cf.}~\cite{stanley:promotion}). These early proofs are all straightforward modifications of those for the standard case. From these results, we observe that Theorem \ref{thm:orbit_size} follows from the claim that $\rot(T) = \evac(T)$ for every $T \in \twinc{n}{k}$. We first saw this approach in \cite{whitemanuscript} for the standard Young tableau case, although similar ideas appear for example in \cite{haiman, stanley:promotion}; we are not sure where it first appeared.  

Finally, beginning at Lemma \ref{lem:row_length_differences}, we prove that for $T \in \twinc{n}{k}$, $\rot(T) = \evac(T)$. Here the situation is more subtle than in the standard case. (For example, we will show that the claim is not generally true for $T$ a rectangular increasing tableau with more than 2 rows.) We proceed by careful analysis of how $\rot, \evac, \dualevac$, and $\pro$ act on the corresponding small Schr\"oder paths.

\begin{remark}
It is not hard to see that K-promotion is reversible, and hence permutes the set of increasing tableaux.
\end{remark}

\begin{lemma}\label{lem:basic_combinatorics_of_P_and_E}
K-evacuation and dual K-evacuation are involutions, $\pro \circ \evac = \evac \circ \pro^{-1}$, and for any increasing tableau $T$, $(\dualevac \circ \evac)(T)= \pro^{\maxv{T}}(T) $.
\end{lemma}

\ytableausetup{boxsize=.33em}
Before proving Lemma \ref{lem:basic_combinatorics_of_P_and_E}, we briefly recall the \emph{K-theory growth diagrams} of \cite[$\mathsection 2,4$]{thomas.yong:K}, which extend the standard Young tableau growth diagrams of S.~Fomin (\emph{cf.}~\cite[Appendix~1]{stanley:ec}). For $T \in \inc{\lambda}{k}$, consider the sequence of Young diagrams $([T]_j)_{0 \leq j \leq |\lambda| - k}$. Note that this sequence of diagrams uniquely encodes $T$. We draw this sequence of Young diagrams horizontally from left to right. Below this sequence, we draw, in successive rows, the sequences of Young diagrams associated to $\pro^i(T)$ for $1 \leq i \leq |\lambda|-k$. Hence each row encodes the K-promotion of the row above it. We offset each row one space to the right. We will refer to this entire array as the \emph{K-theory growth diagram for $T$}. (There are other K-theory growth diagrams for $T$ that one might consider, but this is the only one we will need.) Figure \ref{fig:growth_diagram} shows an example.
We will write $YD_{ij}$ for the Young diagram  $[\pro^{i-1}(T)]_{j-i}$. This indexing is nothing more than imposing ``matrix-style'' or ``English'' coordinates on the K-theory growth diagram. For example in Figure \ref{fig:growth_diagram}, $YD_{58}$ denotes $\ydiagram{2,2}$, the Young diagram in the fifth row from the top and the eighth column from the left.

\begin{figure}[h]
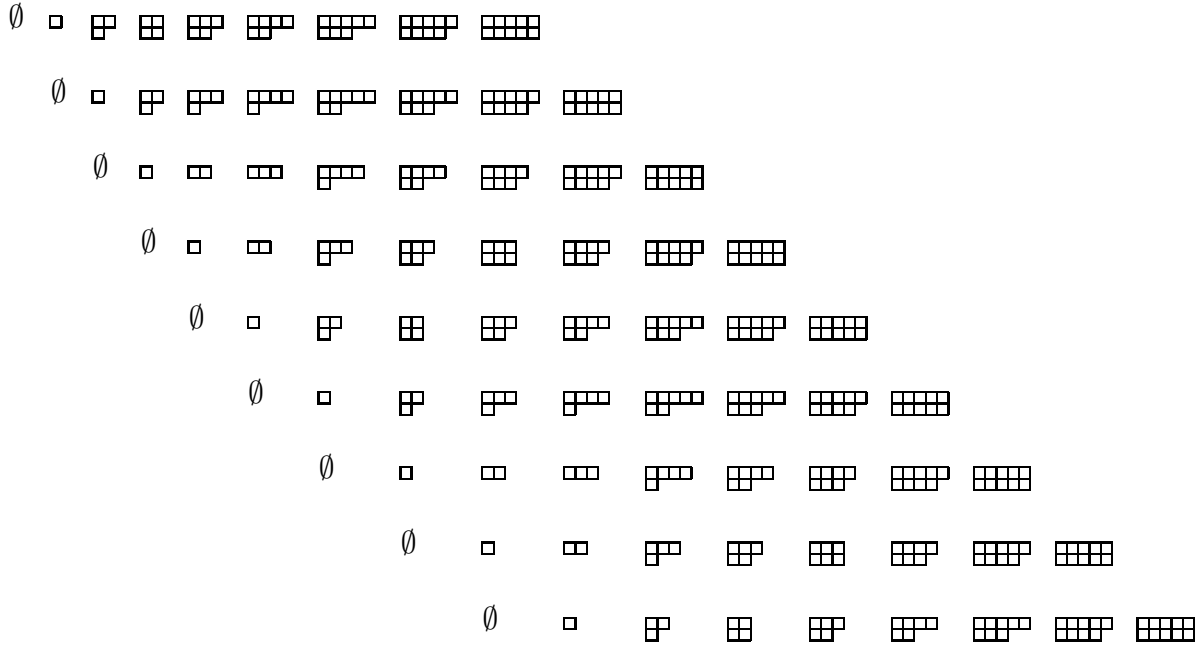

$\begin{array}{lllllllllllllllll} 
\emptyset & \ydiagram{1} & \ydiagram{2,1} & \ydiagram{2,2} & \ydiagram{3,2} & \ydiagram{4,2} & \ydiagram{5,3} & \ydiagram{5,4} & \ydiagram{5,5} \\ \\  
& \emptyset & \ydiagram{1} & \ydiagram{2,1} & \ydiagram{3,1} & \ydiagram{4,1} & \ydiagram{5,2} & \ydiagram{5,3} & \ydiagram{5,4} & \ydiagram{5,5}\\ \\ 
& & \emptyset & \ydiagram{1} & \ydiagram{2} & \ydiagram{3} & \ydiagram{4,1} & \ydiagram{4,2} & \ydiagram{4,3} & \ydiagram{5,4} & \ydiagram{5,5} \\ \\ 
& & & \emptyset & \ydiagram{1} & \ydiagram{2} & \ydiagram{3,1} & \ydiagram{3,2} & \ydiagram{3,3} & \ydiagram{4,3} & \ydiagram{5,4} & \ydiagram{5,5} \\ \\ 
& & & & \emptyset & \ydiagram{1} & \ydiagram{2,1} & \ydiagram{2,2} & \ydiagram{3,2} & \ydiagram{4,2} & \ydiagram{5,3} & \ydiagram{5,4} & \ydiagram{5,5} \\ \\
& & & & & \emptyset & \ydiagram{1} & \ydiagram{2,1} & \ydiagram{3,1} & \ydiagram{4,1} & \ydiagram{5,2} & \ydiagram{5,3} & \ydiagram{5,4} & \ydiagram{5,5} \\ \\
& & & & & & \emptyset & \ydiagram{1} & \ydiagram{2} & \ydiagram{3} & \ydiagram{4,1} & \ydiagram{4,2} & \ydiagram{4,3} & \ydiagram{5,4} & \ydiagram{5,5} \\ \\
& & & & & & & \emptyset & \ydiagram{1} & \ydiagram{2} & \ydiagram{3,1} & \ydiagram{3,2} & \ydiagram{3,3} & \ydiagram{4,3} & \ydiagram{5,4} & \ydiagram{5,5} \\ \\ 
& & & & & & & & \emptyset & \ydiagram{1} & \ydiagram{2,1} & \ydiagram{2,2} & \ydiagram{3,2} & \ydiagram{4,2} & \ydiagram{5,3} & \ydiagram{5,4} & \ydiagram{5,5}
\end{array}$
\caption{The K-theory growth diagram for the tableau $T$ of Figure~\ref{fig:bijections}(A).}
\label{fig:growth_diagram}
\end{figure}

\ytableausetup{boxsize=1.1em}
\begin{remark} \label{rem:growth_rules} \cite[Proposition~2.2]{thomas.yong:K} In any $2 \times 2$ square $\begin{smallmatrix} \lambda & \mu \\ \nu & \xi \end{smallmatrix}$ of Young diagrams in a K-theory growth diagram, $\xi$ is uniquely and explicitly determined by $\lambda, \mu$ and $\nu$. Similarly $\lambda$ is uniquely and explicitly determined by $\mu, \nu$ and $\xi$. Furthermore these rules are symmetric, in the sense that if $\begin{smallmatrix} \lambda & \mu \\ \nu & \xi \end{smallmatrix}$ and $\begin{smallmatrix} \xi & \mu \\ \nu & \rho \end{smallmatrix}$ are both $2 \times 2$ squares of Young diagrams in K-theory growth diagrams, then $\lambda = \rho$.
\end{remark}

\begin{proof}[Proof of Lemma \ref{lem:basic_combinatorics_of_P_and_E}]
Fix a tableau $T \in \inc{\lambda}{k}$.  All of these facts are proven as in the standard case (\emph{cf.}~\cite[$\mathsection 5$]{stanley:promotion}), except one uses K-theory growth diagrams instead of ordinary growth diagrams. We omit some details from these easy arguments. The proof that K-evacuation is an involution appears in greater detail as \cite[Theorem 4.1]{thomas.yong:K}. For rectangular shapes, the fact that dual K-evacuation is an involution follows from the fact that K-evacuation is, since $\dualevac = \rot  \circ \evac \circ \rot$. 

Briefly one observes the following. Essentially by definition, the central column (the column containing the rightmost $\emptyset$) of the K-theory growth diagram for $T$ encodes the K-evacuation of the first row as well as the dual K-evacuation of the last row. The first row encodes $T$ and the last row encodes $\pro^{|\lambda|-k}(T)$. Hence $\evac(T) = \dualevac(\pro^{|\lambda|-k}(T))$.

By the symmetry mentioned in Remark \ref{rem:growth_rules}, one also observes that the first row encodes the K-evacuation of the central column and that the last row encodes the dual K-evacuation of the central column. This yields $\evac(\evac(T)) = T$ and $\dualevac(\dualevac(\pro^{|\lambda|-k}(T))) = \pro^{|\lambda|-k}(T)$, showing that K-evacuation and dual K-evacuation are involutions. Combining the above observations, yields $(\dualevac \circ \evac)(T)= \pro^{|\lambda|-k}(T) $.

Finally to show $\pro \circ \evac = \evac \circ \pro^{-1}$, it is easiest to append an extra $\emptyset$ to the lower-right of the diagonal line of $\emptyset$s that appears in the K-theory growth diagram. This extra $\emptyset$ lies in the column just right of the central one. This column now encodes the K-evacuation of the second row. Hence by the symmetry mentioned in Remark \ref{rem:growth_rules}, the K-promotion of this column is encoded by the central column. Thus if $S = \pro(T)$, the central column encodes $\pro(\evac(S))$. But certainly $\pro^{-1}(S) =T$ is encoded by the first row, and we have already observed that the central column encodes $\evac(T)$. Therefore $\pro(\evac(S)) = \evac(\pro^{-1}(S)).$
\end{proof}

Let $\er(T)$ be the least positive integer such that $(\dualevac \circ \evac)^{\er(T)}(T) = T$. We call this number the \emph{evacuation rank} of $T$. Similarly we define the \emph{promotion rank} $\pr(T)$ to be the least positive integer such that $\pro^{\pr(T)}(T) = T$.

\begin{corollary}\label{cor:evacuation_rank}
Let $T$ be a increasing tableau. Then $\er(T)$ divides $\pr(T)$, $\pr(T)$ divides $\maxv{T} \cdot \er(T)$, and the following are equivalent:
	\begin{itemize}
			\item[(a)] $\evac(T) = \dualevac(T)$,
			\item[(b)] $\er(T) = 1$, 
			\item[(c)] $\pr(T)$ divides $\maxv{T}$. 
	\end{itemize}
	
	Moreover if $T$ is rectangular and $\evac(T) = \rot(T)$, then $\evac(T) = \dualevac(T)$.
\end{corollary}
\begin{proof}
Since, by Lemma \ref{lem:basic_combinatorics_of_P_and_E}, we have $(\dualevac \circ \evac)(T)= \pro^{\maxv{T}}(T) $, the evacuation rank of $T$ is the order of $c^{\maxv{T}}$ in the cyclic group $\mathcal{C}_{\pr(T)}=\langle c \rangle$. In particular, $\er(T)$ divides $\pr(T)$. Since $T = (\dualevac \circ \evac)^{\er(T)}(T) = (\pro^{\maxv{T}})^{\er(T)}(T) = \pro^{\maxv{T}\cdot\er(T)}(T)$, we have that $\maxv{T} \cdot \er(T)$ is a multiple of $\pr(T)$.

The equivalence of (a) and (b) is immediate from dual evacuation being an involution. These imply (c), since $(\dualevac \circ \evac)(T)= \pro^{\maxv{T}}(T) $. If $\pr(T)$ divides $\maxv{T}$, then $\pro^{\maxv{T}}(T) = T$, so $(\dualevac \circ \evac)(T) = T$, showing that (c) implies (b).

By definition, for rectangular $T$, $\dualevac(T) =  (\rot  \circ \evac \circ \rot)(T)$, so if $\rot(T) = \evac(T)$, then $\dualevac(T) = (\evac \circ \evac \circ \evac)(T) = \evac(T)$.
\end{proof}
Thus to prove Theorem~\ref{thm:orbit_size}, it suffices to show the following proposition:
\begin{proposition}\label{prop:evac=rot}
Let $T \in \twinc{n}{k}$. Then $\evac(T) = \rot(T)$.
\end{proposition}
We will also need Proposition~\ref{prop:evac=rot} in the proof of Theorem~\ref{thm:evacuation_CSP}, and additionally it has recently found application in \cite[Theorem~5.3]{bloom.p.saracino} which demonstrates homomesy (as defined by \cite{propp.roby:fpsac}) on $\twinc{n}{k}$. 
To prove Proposition~\ref{prop:evac=rot}, we use the bijection between $\twinc{n}{k}$ and small Schr\"oder paths from Theorem \ref{thm:enumeration}. These paths are themselves in bijection with the sequence of their node heights, which we call the \emph{height word}. Figure \ref{fig:bijections}(C) shows an example. For $T \in \twinc{n}{k}$, we write $P_T$ for the corresponding small Schr\"oder path and $S_T$ for the corresponding height word.

\begin{lemma}\label{lem:row_length_differences}
For $T \in \twinc{n}{k}$, the $i$th letter of the height word $S_T$ is the difference between the lengths of the first and second rows of the Young diagram $[T]_{i-1}$.
\end{lemma}
\begin{proof}
By induction on $i$. For $i=1$, both quantities equal 0. The $i$th segment of $P_T$ is an upstep if and only if $[T]_{i} \, \backslash \, [T]_{i-1}$ is a single box in the first row. The $i$th segment of $P_T$ is an downstep if and only if $[T]_{i} \, \backslash \, [T]_{i-1}$ is a single box in the second row. The $i$th segment of $P_T$ is horizontal if and only if $[T]_{i} \, \backslash \, [T]_{i-1}$ is two boxes, one in each row. 
\end{proof}

\begin{lemma}\label{lem:rot_is_reflection_of_path}
Let $T \in \twinc{n}{k}$. Then $P_{\rot(T)}$ is the reflection of $P_T$ across a vertical line and $S_{\rot(T)}$ is the word formed by reversing $S_T$.
\end{lemma}
\begin{proof}
Rotating $T$ by $180$ degrees corresponds to reflecting $P_T$ across the horizontal axis. Reversing the alphabet corresponds to rotating $P_T$  by $180$ degrees. Thus $\rot(T)$ corresponds to the path given by reflecting $P_T$ across a vertical line. 

The correspondence between reflecting $P_T$ and reversing $S_T$ is clear.
\end{proof}

\begin{lemma}\label{lem:evacuation_word_formulation}
Let $T \in \twinc{n}{k}$ and $M =2n-k$. Let $x_{i}$ denote the $(M+2-i)$th letter of the height word $S_{\pro^{i-1}(T)}$. Then $S_{\evac(T)} = x_{M+1} x_{M} \dots x_1$.
\end{lemma}
\begin{proof}
Consider the K-theory growth diagram for $T$. Observe that $YD_{i,M+1}$ is the $(M+2-i)$th Young diagram in the $i$th row. Hence by Lemma \ref{lem:row_length_differences}, $x_i$ is the difference between the lengths of the rows of $YD_{i,M+1}$. But $YD_{i,M+1}$ is also the $i$th Young diagram from the top in the central column. The lemma follows by recalling that the central column encodes $\evac(T)$. 
\end{proof}

We define the \emph{flow path} $\phi(T)$ of an increasing tableau $T$ to be the set of all boxes that are ever empty during the K-promotion that forms $\pro(T)$ from $T$.

\begin{lemma}\label{lem:promotion_word_algorithm}
Let $T \in \twinc{n}{k}$.
\begin{enumerate}[(a)]
\item The word $S_T$ may be written in exactly one way as $0w_10w_3$ or $0w_11w_20w_3$, where $w_1$ is a sequence of strictly positive integers that ends in 1 and contains no consecutive 1s, $w_2$ is a (possibly empty) sequence of strictly positive integers, and $w_3$ is a (possibly empty) sequence of nonnegative integers. 

\item Let $w_1^-$ be the sequence formed by decrementing each letter of $w_1$ by 1. Similarly, let $w_3^+$ be formed by incrementing each letter of $w_3$ by 1. 

If $S_T$ is of the form $0w_10w_3$, then $S_{\pro(T)} =  w_1^-1w_3^+0$. If $S_T$ is of the form $0w_11w_20w_3$, then $S_{\pro(T)} =  w_1^-1w_21w_3^+0$.
\end{enumerate}
\end{lemma}
\begin{proof}
It is clear that $S_T$ may be written in exactly one of the two forms. Write $\ell_i$ for the length of $w_i$. Suppose first that $S_T$ is of the form $0w_10w_3$. By the correspondence between tableaux and height sequences, $[T]_{\ell_1 + 1}$ is a rectangle, and for no $0<x < \ell_1 + 1$ is $[T]_x$ a rectangle. Say $[T]_{\ell_1 + 1} = (m, m)$. The flow path $\phi(T)$ contains precisely the first $m$ boxes of the first row and the last $n - m + 1$ boxes of the second row. Only the entry in box $(2,m)$ changes row during K-promotion. It is clear then that $S_{\pro(T)} =  w_1^-1w_3^+0$.

Suppose now that $S_T$ is of the form $0w_11w_20w_3$. Then $[T]_{\ell_1 + 1} = (p+1, p)$ for some $p$, and $[T]_{\ell_1 + \ell_2 + 2} = (m, m)$ for some $m$. The flow path $\phi(T)$ contains precisely the first $m$ boxes of the first row and the last $n - p + 1$ boxes of the second row. It is clear then that $S_{\pro(T)} =  w_1^-1w_21w_3^+0$.
\end{proof}

Notice that when $T \in \syt{2 \times n}$, $S_T$ can always be written as $0w_10w_3$. Hence by Lemma~\ref{lem:promotion_word_algorithm}(b), the promotion $S_{\pro(T)}$ takes the particularly simple form $w_1^-1w_3^+0$.

For $T \in \twinc{n}{k}$, take the first $2n-k+1$ columns of the K-theory growth diagram for $T$. Replace each Young diagram in the resulting array by the difference between the lengths of its first and second rows. Figure \ref{fig:equality_of_first_row_last_column_in_hgd} shows an example. We write $a_{ij}$ for the number corresponding to the Young diagram $YD_{ij}$. By Lemma \ref{lem:row_length_differences}, we see that the $i$th row of this array of nonnegative integers is exactly the first $2n-k+2- i$ letters of $S_{\pro^{i-1}(T)}$. Therefore we will refer to this array as the \emph{height growth diagram for $T$}, and denote it by $\hgd{T}$. Observe that the rightmost column of $\hgd{T}$ corresponds to the central column of the K-theory growth diagram for $T$.

\begin{figure}[h]
$\begin{array}{ccccccccc} 0 & 1 & 1 & 0 & 1 & 2 & 2 & 1 & 0 \\ & 0 & 1 & 1 & 2 & 3 & 3 & 2 & 1 \\ 
& & 0 & 1 & 2 & 3 & 3 & 2 & 1 \\ & & & 0 & 1 & 2 & 2 & 1 & 0 \\ & & & & 0 & 1 & 1 & 0 & 1 \\ & & & & & 0 & 1 & 1 & 2 \\ & & & & & & 0 & 1 & 2 \\ & & & & & & & 0 & 1 \\ & & & & & & & & 0 
\end{array}$
\caption{The height growth diagram $\hgd{T}$ for the tableau $T$ shown in Figure \ref{fig:bijections}(A). The $i$th row shows the first $10 - i$ letters of $S_{\pro^{i-1}(T)}$. Lemma \ref{lem:equality_of_first_row_last_column_in_hgd} says that row~1 is the same as column 9, read from top to bottom.}\label{fig:equality_of_first_row_last_column_in_hgd}
\end{figure}

We will sometimes write $\pro(S_T)$ for $S_{\pro(T)}$. 

\begin{lemma}\label{lem:equality_of_first_row_last_column_in_hgd}
In $\hgd{T}$ for $T \in \twinc{n}{k}$, we have for all $j$ that $a_{1j} = a_{j,2n-k+1}$.
\end{lemma}
\begin{proof}
Let $M = 2n-k$. We induct on the length of the height word. (The length of $S_T$ is $M+1$.) 
\begin{case} The height word $S_T$ contains an internal 0. \end{case} 
Let the first internal 0 be the $j$th letter of $S_T$. Then the first $j$ letters of $S_T$ are themselves the height word of some smaller rectangular increasing tableau $T'$. Because of the local properties of K-theory growth diagrams mentioned in Remark \ref{rem:growth_rules}, we observe that the $j$th column of $\hgd{T}$ is the same as the rightmost column of $\hgd{T'}$. The height word $S_{T'}$ is shorter than the height word $S_T$, so by inductive hypothesis, the first row of $\hgd{T'}$ is the same as its rightmost column, read from top to bottom. Thus in $\hgd{T}$, the first $j$ letters of row 1 are the same as column $j$. 

According to Lemma \ref{lem:promotion_word_algorithm}(b), in each of the first $j$ rows of $\hgd{T}$, the letter in column $j$ is less than or equal to all letters to its right.  Furthermore the letters in columns $j + 1$ through $M+1$ are incremented, decremented, or unchanged from one row to the next in exactly the same way as the letter in column $j$. That is to say, for any $g\leq j \leq h$, $a_{gh} - a_{1h} = a_{gj} - a_{1j}$. Since $a_{1j} = a_{1,M+1} = 0$, this yields $a_{gj} = a_{g,M+1}$, so column $j$ is the same as the first $j$ letters of column $M+1$, read from top to bottom. Thus the first $j$ letters of row 1 are the same as the first $j$ letters of column $M+1$.

Now since $a_{j,M+1} = 0$, row $j$ of $\hgd{T}$ is itself the height word of some smaller tableau $T^\dagger$. Again by inductive hypothesis, we conclude that in $\hgd{T}$, row $j$ is the same as the last $M + 2-j$ letters of column $M+1$, read from top to bottom. But as previously argued, the letters in columns $j + 1$ through $M+1$ are incremented, decremented, or unchanged from one row to the next in the same way as the letter in column $j$. Hence row $j$ agrees with the last $M + 2-j$ letters of row 1, and so the last $M+2-j$ letters of row 1 agree with the last $M+2-j$ letters of column $M+1$. Thus, as desired, row 1 of $\hgd{T}$ is the same as column $M+1$, read from top to bottom.

\begin{case} The height word $S_T$ contains no internal 0. \end{case}
Notice that $s_{1M} = 1$. Hence by Lemma \ref{lem:promotion_word_algorithm}(b), there will be an internal 0 in the K-promotion $S_{\pro(T)}$, unless $S_T$ is the word $010$ or begins $011$. 

\begin{subcase} The height word $S_{\pro(T)}$ contains an internal 0. \end{subcase}
Let the first internal 0 be in column $j$ of $\hgd{T}$. Then by Lemma \ref{lem:promotion_word_algorithm}(b), the first $j -1$ letters of row 2 of $\hgd{T}$ are all  exactly one less than the letters directly above them in row 1. That is for $2\leq h\leq j$, we have $a_{2h} = a_{1h} -1$. Also observe $a_{2,M+1}=1$.
 
The first $j-1$ letters of row 2 are the height sequence of some tableau $T'$ with $S_{T'}$ shorter than $S_T$. So by inductive hypothesis, the first $j - 1$ letters of row 2 of $\hgd{T}$ are the same as the last $j-1$ letters of column $j$, read from top to bottom. That is to say $a_{2h} = a_{hj}$, for all $2 \leq h \leq j$. 

Since $a_{2j}=0$ and $a_{2,M+1}=1$, and since  the letters below the first row in columns $j + 1$ through $M+1$ are incremented, decremented, or unchanged in the same way as the letter in column $j$, it follows that for $2\leq h \leq j$, $a_{hj} = a_{h,M+1} - 1$. Therefore the first $j$ letters of row 1 are the same as the first $j$ letters of column $M+1$.

Consider the height word $S'$ formed by prepending a 0 to the last $M+2 -j$ letters of row 1. The last $M+2 -j$ letters of row 2 are the same as the first $M+2-j$ letters of $\pro(S')$. But the last $M+2 -j$ letters of row 2 are the same as row $j$. Therefore by inductive hypothesis, the last $M+2 -j$ letters of column $M+1$ are the same as the last $M+2-j$ letters of $S'$, which are by construction exactly the last $M+2 -j$ letters of row 1. Thus row 1 is exactly the same as column $M+1$.

\begin{subcase} $S_T= 010$. \end{subcase}Trivially verified by hand. 

\begin{subcase} $S_T$ begins $011$. \end{subcase} Row 2 of $\hgd{T}$ is produced from row 1 by deleting the initial 0, changing the first 1 into a 0, and changing the final 0 into a 1. Let $S'$ be the word formed by replacing the final 1 of row 2 with a 0. Note that $a_{3,M+1} = 1$. Therefore row 3 agrees with the first $M-1$ letters of $\pro(S')$. Therefore by inductive hypothesis, the last $M-1$ letters of $S'$ are the same as the last $M-1$ letters of column $M+1$. Hence the last $M$ letters of column $M+1$ are the same as row 2, except for having a 1 at the beginning instead of a 0 and a 0 at the end instead of a 1. But these are exactly the changes we made to produce row 2 from row 1. Thus row 1 is the same as column $M+1$.
\end{proof}

\begin{corollary}\label{cor:tableau_path_word}
In the notation of Lemma \ref{lem:evacuation_word_formulation}, $S_T = x_1 x_2 \dots x_{M+1}$. \qed
\end{corollary}

\begin{proof}[Proof of Proposition~\ref{prop:evac=rot}]
By Corollary \ref{cor:tableau_path_word}, $S_T = x_1 x_2 \dots x_{2n-k+1}$. Hence by Lemma \ref{lem:rot_is_reflection_of_path}, we have $S_{\rot(T)} = x_{2n-k+1} x_{2n-k} \dots x_1$. However Lemma \ref{lem:evacuation_word_formulation} says also $S_{\evac(T)} = x_{2n-k+1} x_{2n-k} \dots x_1$. By the bijective correspondence between tableaux and height words, this yields $\evac(T) = \rot(T)$.
\end{proof}

This completes our first proof of Theorem \ref{thm:orbit_size}. We will obtain alternate proofs in Sections~\ref{sec:tropical_frieze} and \ref{sec:CSP}. We now show a counterexample to the obvious generalization of Theorem~\ref{thm:orbit_size} to increasing tableaux of more than two rows.
\begin{example}\label{ex:counterexample_to_orbit_size_conjecture}
If $T$ is the increasing tableau \ytableausetup{centertableaux} $\ytableaushort{1247,3{\underline{5}}68,5{\underline{7}}8{10},79{10}{11}},$ then $\pro^{11}(T) =  \ytableaushort{1247,3{\underline{4}}68,5{\underline{6}}8{10},79{10}{11}}$. (The underscores mark entries that differ between the two tableaux.) It can be verified that the promotion rank of $T$ is 33. 
\end{example}
Computer checks of small examples (including all with at most seven columns) did not identify such a counterexample for $T$ a 3-row rectangular increasing tableau. However it is not generally true that $\evac(T) = \rot(T)$ for $T \in \inc{3 \times n}{k}$.
\begin{example}\label{ex:counterexample_to_rot=evac}
If $T$ is the increasing tableau \ytableausetup{centertableaux} $\ytableaushort{124,3{\underline{4}}6,578}$ then $\evac(T) = T$, while $\rot(T) =  \ytableaushort{124,3{\underline{5}}6,578}$. Nonetheless the promotion rank of $T$ is $2$, which divides $8$, so the obvious generalization of Theorem~\ref{thm:orbit_size} holds in this example.
\end{example}

\section{Tropical frieze patterns}\label{sec:tropical_frieze}
In this section, we make connections with tropical frieze patterns, which we use to give an alternate proof of Proposition~\ref{prop:evac=rot} and Theorem~\ref{thm:orbit_size}. 

Frieze patterns are simple cluster algebras introduced in \cite{conway.coxeter}. They are infinite arrays of real numbers bounded between two parallel diagonal lines of $1$s, satisfying the property that for each $2 \times 2$ subarray $\begin{smallmatrix} a & b \\ c & d \end{smallmatrix}$ the relation $d = (bc + 1) / a$ holds. Figure~\ref{fig:classical_frieze} shows an example. Notice that by this local defining relation, the frieze pattern is determined by any one of its rows. 

\begin{figure}[h]
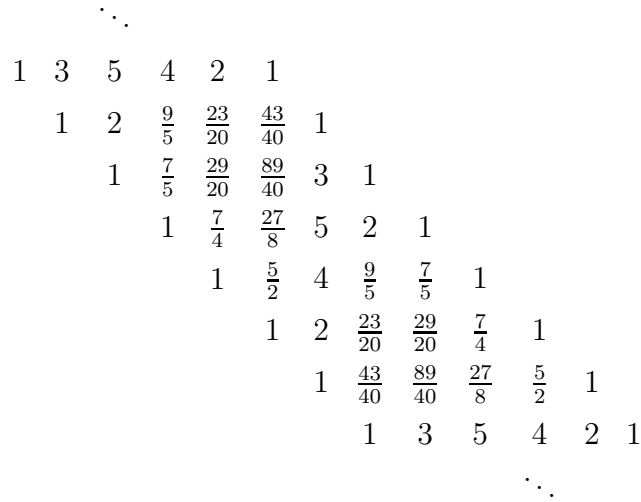

\begingroup
\renewcommand*{\arraystretch}{1.4}
$$\begin{array}{ccccccccccccc}
 &  &  \ddots & & &  & & & & & &&\\
1 & 3 & 5 & 4 & 2 & 1 & & & & & &&\\
& 1 & 2 & \frac{9}{5} &\frac{23}{20} & \frac{43}{40} & 1 & & &&& &\\
& & 1 & \frac{7}{5} & \frac{29}{20} & \frac{89}{40} & 3 & 1 & && &&\\
& & & 1 & \frac{7}{4} & \frac{27}{8} & 5 & 2 & 1 & & &&\\
& & & & 1 & \frac{5}{2} & 4 & \frac{9}{5} & \frac{7}{5} & 1&& & \\
& & & & & 1 & 2 & \frac{23}{20} & \frac{29}{20} & \frac{7}{4} & 1& &\\
& & & & & & 1 &  \frac{43}{40} & \frac{89}{40} & \frac{27}{8} & \frac{5}{2} & 1 & \\
& & & & & & & 1 & 3 & 5 & 4 & 2 & 1 \\
 &  &  & & &  & & & & & \ddots &&
\end{array}$$
\endgroup \caption{A classical Conway--Coxeter frieze pattern.} \label{fig:classical_frieze}
\end{figure}

A tropical analogue of frieze patterns may be defined by replacing the bounding 1s by 0s and imposing the tropicalized relation $d = \max \left( b + c, 0 \right) - a$ on each $2 \times 2$ subarray. Such \emph{tropical frieze patterns} have attracted some interest lately (e.g.,~\cite{propp, guo:frieze, assem.dupont, grabowski}).

One of the key results of \cite{conway.coxeter} is that, if rows of a frieze pattern have length $\ell$, then each row is equal to the row $\ell + 1$ rows below it, as well as to the central column between these two rows, read from top to bottom (\emph{cf.}~Figure~\ref{fig:classical_frieze}). That the same periodicity occurs in tropical frieze patterns may be proved directly by imitating the classical proof, or it may be easily derived from the classical periodicity by taking logarithms. We do the latter.

\begin{lemma}
If $\mathcal{TF}$ is a tropical frieze diagram with rows of length $\ell$, then each row is equal to the row $\ell + 1$ rows below it, as well as to the central column between these two rows, read from top to bottom.
\end{lemma}
\begin{proof}
Pick a row $R = (a_0, a_1, \dots, a_\ell)$ of $\mathcal{TF}$. Let $e$ be the base of the natural logarithm.

For $t \in \R_{> 0}$, construct the classical frieze pattern containing the row $R' = (e^{a_0/t}, \dots, e^{a_n/t})$. Now take the logarithm of each entry of this frieze pattern and multiply each entry by $t$. Call the result $\mathcal{F}_t$. Note that $\mathcal{F}_t$ is not in general a frieze pattern; however, it does have the desired periodicity. Also observe that the row $R$ appears in each $\mathcal{F}_t$ as the image of $R'$. Now take $\lim_{t \to 0} \mathcal{F}_t$. This limit also contains the row $R$. This process converts the relation $d = (bc + 1) / a$ into the relation $d = \max \left( b + c, 0 \right) - a$, so $\lim_{t \to 0} \mathcal{F}_t = \mathcal{TF}$. Since each $\mathcal{F}_t$ has the desired periodicities, so does $\mathcal{TF}$.
\end{proof}

Let $T \in \twinc{n}{k}$. Recall from Section~\ref{sec:schroeder_paths} the K-theory growth diagram for $T$. Replace each Young diagram by 1 less than the difference between the lengths of its first and second rows. Delete the first and last number in each row (necessarily $-1$). We call the resulting array the \emph{jeu de taquin frieze pattern} of $T$. (It is obviously closely related to the height growth diagram.) Each row is an integer sequence encoding the tableau corresponding to that row of the K-theory growth diagram. (Indeed it is the height word with all terms decremented by 1, and the first and last terms removed.)

\begin{remark}\label{rem:tropical_friezes_corresponding_to_tableaux}
Observe that this map from increasing tableaux to integer sequences is injective. The image is exactly those sequences such that 
\begin{itemize}
\item[(1)] the first and last terms are 0,
\item[(2)] every term is $\geq -1$, 
\item[(3)] successive terms differ by at most 1, and
\item[(4)] there are no consecutive $-1$s.
\end{itemize} 

An integer sequence is the image of a standard Young tableau if it satisfies the stronger
\begin{itemize}
\item[($3'$)] successive terms differ by exactly 1,
\end{itemize} in place of condition (3).
\end{remark} 

\begin{example}\label{ex:tropical_frieze}
\ytableausetup{boxsize=1.1em}
For $T = \ytableaushort{1235,4567}$, we obtain the jeu de taquin frieze pattern

\[\begin{array}{ccccccccccccc}
0 & 1 & 2 & 1 & 1 & 0 & & & & &\\
& 0 & 1 & 0 & 0 & -1 & 0 & & & &\\
& & 0 & -1 & 0 & 0 & 1 & 0 & & &\\
& & & 0 & 1 & 1 & 2 & 1 & 0 & & \\
& & & & 0 & 0 & 1 & 0 & -1 & 0 & \\
& & & & & 0 & 1 & 0 & 0 & 1 & 0 \\
& & & & & & 0 & -1 & 0 & 1 & 0 & 0\\
& & & & & & & 0 & 1 & 2 & 1 & 1 & 0
\end{array}.\]
\end{example}

The fact that the first row, last row, and central column are all equal is equivalent to Lemma~\ref{lem:equality_of_first_row_last_column_in_hgd} and Theorem~\ref{thm:orbit_size}. The following lemma gives an alternate approach.

\begin{lemma}
A jeu de taquin frieze pattern is a subarray of a tropical frieze pattern.
\end{lemma}
\begin{proof}
It is clear that we have bounding diagonals of 0s. It suffices to verify the local defining relation $d = \max \left( b + c, 0 \right) - a$ on $2 \times 2$ subarrays. This follows fairly easily from the algorithmic relation of Lemma~\ref{lem:promotion_word_algorithm}.
\end{proof}

The next corollary follows immediately from the above; although it can be proven directly, the derivation from results on K-promotion seems more enlightening.

\begin{corollary}\label{cor:well_behaved_friezes}
Let $\mathcal{TF}$ be a tropical frieze diagram. 
\begin{itemize}
\item[(a)] If any row of $\mathcal{TF}$ satisfies the conditions (1), (2), ($3'$) of Remark~\ref{rem:tropical_friezes_corresponding_to_tableaux}, then every row of $\mathcal{TF}$ does.
\item[(b)]If any row of $\mathcal{TF}$ satisfies the four conditions (1), (2), (3), (4) of Remark~\ref{rem:tropical_friezes_corresponding_to_tableaux}, then every row of $\mathcal{TF}$ does.
\end{itemize}
 \qed
\end{corollary}

\begin{remark}
We speculate ahistorically that one could have discovered K-promotion for increasing tableaux in the following manner. First one could have found a proof of Theorem~\ref{thm:orbit_size} for standard Young tableaux along the lines of this section (indeed similar ideas appear in \cite{kirillov.berenstein}), proving Corollary~\ref{cor:well_behaved_friezes}(a) in the process. Looking for similar results, one might observe Corollary~\ref{cor:well_behaved_friezes}(b) experimentally and be lead to discover K-promotion in proving it.

Are there other special tropical frieze patterns hinting at a promotion theory for other classes of tableaux?
For example, tropical friezes containing a row satisfying conditions (1), (2), (3) of Remark~\ref{rem:tropical_friezes_corresponding_to_tableaux} seem experimentally to be well-behaved, with all rows having successive terms that differ by at most 2.
\end{remark}

We are able to prove the order of promotion on $\syt{3 \times n}$ in a similar fashion, using tropicalizations of the 2-frieze patterns of \cite{morier-genoud.ovsienko.tabachnikov}. Unfortunately we have been unable to extend this argument to $\inc{3 \times n}{k}$ for $k>0$. Example~\ref{ex:counterexample_to_rot=evac} suggests that such an extension would be difficult.

\section{Representation-theoretic interpretations}\label{sec:representation_theory}
In \cite{stembridge:evacuation}, J.~Stembridge proved that, for every $\lambda$, $(\syt{\lambda}, \mathcal{C}_2, f^\lambda(q))$ exhibits cyclic sieving, where the non-identity element of $\mathcal{C}_2$ acts by evacuation and $f^\lambda(q)$ is the standard $q$-analogue of the hook-length formula. We briefly recall the outline of this argument. Considering the Kazhdan--Lusztig cellular basis for the Specht module $V^\lambda$, the long element $w_0 \in S_{|\lambda|}$ acts (up to a controllable sign) on $V^\lambda$ by permuting the basis elements. Moreover under a natural indexing of the basis by $\syt{\lambda}$, the permutation is exactly evacuation. The cyclic sieving then follows by evaluating the character of $V^\lambda$ at $w_0$.

We can give an analogous proof of Theorem~\ref{thm:evacuation_CSP}. However to avoid redundancy we do not do so here, and instead derive Theorem~\ref{thm:evacuation_CSP} by direct reduction to J.~Stembridge's result. 

Let $\flag$ denote the map from $\twinc{n}{k}$ to $\syt{n-k, n-k, 1^k}$ from Proposition~\ref{prop:flag_bijection}. Theorem~\ref{thm:evacuation_CSP} follows immediately from the following proposition combined with J.~Stembridge's previously-described result.

\begin{proposition}\label{prop:evacuation_is_evacuation}
For all $T \in \twinc{n}{k}$, $\evac(\flag(T)) = \flag(\evac(T))$.
\end{proposition}

Proposition~\ref{prop:evacuation_is_evacuation} was first suggested to the author by B.~Rhoades, who also gave some ideas to the proof. Before we prove this result, we introduce some additional notation. Observe that if $S \in \syt{n-k, n-k, 1^k}$, then $\rot(S)$ has skew shape $((n-k)^{k+2}) / ((n-k-1)^k)$. 
Let $T \in \twinc{n}{k}$. Let $A$ be the set of numbers that appear twice in $T$. Let $C$ be the set of numbers that appear in the first row immediately left of an element of $A$. 
Let $d(T)$ be the tableau of shape $(n-k,n-k)$ formed by deleting all elements of $A$ from the second row of $T$ and all elements of $C$ from the first. 
A tableau $\skewflag(T)$ of skew shape $((n-k)^{k+2}) / ((n-k-1)^k)$ is given by attaching $C$ in the $k$th column of $d(T)$. Figure~\ref{fig:skew_flag} illustrates these maps.
It is immediate from comparing the definitions that $\rot(\flag(T)) = \skewflag(\rot(T)).$ 

\begin{figure}[h]
	\begin{subfigure}[b]{0.45\textwidth}
		\centering
		\ytableaushort{12347,45678}
		\caption{Increasing tableau $T$}
	\end{subfigure}
	\begin{subfigure}[b]{0.45\textwidth}
		\centering
		\ytableaushort{123,467,5,8}
		\caption{Flag-shaped tableau $\flag(T)$}
	\end{subfigure} 
	\begin{subfigure}[b]{0.45\textwidth}
		\centering
		\ytableaushort{127,568}
		\caption{$d(T)$}
	\end{subfigure}
	\begin{subfigure}[b]{0.45\textwidth}
		\centering
		\ytableaushort{\none \none 3, \none \none 4, 127,568}
		\caption{$\skewflag(T)$}
	\end{subfigure}
	\caption{An illustration of the maps $\flag$, $d$, and $\skewflag$ on 2-row rectangular increasing tableaux.}\label{fig:skew_flag}
\end{figure}

\begin{lemma}\label{lem:skew_flag_rectification}
For all $T \in \twinc{n}{k}$, $\flag(T)$ is the rectification of $\skewflag(T)$.
\end{lemma}
\begin{proof}
It is enough to show that $\flag(T)$ and $\skewflag(T)$ lie in the same plactic class. Consider the row reading word of $\skewflag(T)$. Applying the RSK algorithm to the first $2(n-k)$ letters of this word, we obtain the tableau $d(T)$. The remaining letters are those that appear in the first row of $T$ immediately left of a element that appears in both rows. These letters are in strictly decreasing order. It remains to Schensted bump these remaining letters into $d(T)$ in strictly decreasing order, and observe that we obtain the tableau $\flag(T)$.

Suppose we first bump in the letter $i$. By assumption $i$ appears in the first row of $T$ immediately left of an element $j$ that appears in both rows. Since $i$ is the biggest such letter, $j$ appears in the first row of $d(T)$. Hence $i$ bumps $j$ out of the first row. We then bump $j$ into the second row. The element that it bumps out of the second row is the least element greater than it. This is precisely the element $h$ immediately to the right of $j$ in the second row of $T$. 

Repeating this process, since the elements we bump into $d(T)$ are those that appear in the first row of $T$ immediately left of elements that appear in both rows, we observe that the elements that are bumped out of the first row are precisely those that appear in both rows of $T$. Hence the first row of the resulting tableau consists exactly of those elements that appear only in the first row of $T$. Thus the first row of the rectification of $\skewflag(T)$ is the same as the first row of $\flag(T)$.

Moreover, since the elements bumped out of the first row are precisely those that appear in both rows of $T$, these are also exactly the elements bumped into the second row. Therefore the elements bumped out of the second row are exactly those that appear in the second row of $T$ immediately right of an element that appears twice. Thus the second row of the rectification of $\skewflag(T)$ is also the same as the second row of $\flag(T)$.

Finally, since elements are bumped out of the second row in strictly decreasing order, the resulting tableau has the desired shape $(n-k, n-k, 1^k)$. Thus $\flag(T)$ is the rectification of $\skewflag(T)$.
\end{proof}

\begin{proof}[Proof of Proposition~\ref{prop:evacuation_is_evacuation}]
Fix $T \in \twinc{n}{k}$. Recall from Proposition~\ref{prop:evac=rot} that $\evac(T) = \rot(T)$. Evacuation of standard Young tableaux can be defined as applying $\rot$ followed by rectification to a straight shape. Hence to prove $\evac(\flag(T)) = \flag(\evac(T))$, it suffices to show that $\flag(\rot(T))$ is the rectification of $\rot(\flag(T))$. We observed previously that $\rot(\flag(T)) = \skewflag(\rot(T))$. Hence Lemma \ref{lem:skew_flag_rectification} completes the proof by showing that $\skewflag(\rot(T))$ rectifies to $\flag(\rot(T))$.
\end{proof}

B.~Rhoades' proof of cyclic sieving for $\syt{m \times n}$ under promotion follows the same general structure as J.~Stembridge's proof for $\syt{\lambda}$ under evacuation. That is, he considers the Kazhdan--Lusztig cellular representation $V^{m \times n}$ with basis indexed by $\syt{m \times n}$ and looks for an element $w \in S_{mn}$ that acts (up to scalar multiplication) by sending each basis element to its promotion. It turns out that the long cycle $w = (123\dots mn)$ suffices.

Given our success interpreting Theorem~\ref{thm:evacuation_CSP} along these lines, one might hope to prove Theorem~\ref{thm:promotion_CSP} as follows. Take the Kazhdan--Lusztig cellular representation $V^{(n-k,n-k,1^k)}$ and index the basis by $\twinc{n}{k}$ via the bijection of Proposition~\ref{prop:flag_bijection}. Then look for an element $w \in S_{2n-k}$ that acts (up to scalar multiplication) by sending each basis element to its K-promotion. With B.~Rhoades, the author investigated this approach. Unfortunately we found by explicit computation that $w_0$ is generally the only element of $S_{2n-k}$ acting on $V^{(n-k,n-k,1^k)}$ as a permutation matrix (even up to scalar multiplications). This does not necessarily mean that no element of the group algebra could play the role of $w$. However prospects for this approach seem to us dim. We prove Theorem~\ref{thm:promotion_CSP} in the next section by elementary combinatorial methods.

\section{Proof of Theorem \ref{thm:promotion_CSP}}\label{sec:CSP}
Recall definition (\ref{eq:f}) of $f(q)$. Our strategy (modeled throughout on \cite[$\mathsection$7]{reiner.stanton.white}) is to explicitly evaluate $f$ at roots of unity and compare the result with a count of increasing tableaux. To count tableaux, we use the bijection with noncrossing partitions given in Proposition \ref{prop:partition_bijection}. We will find that the symmetries of these partitions more transparently encode the promotion ranks of the corresponding tableaux.

\begin{lemma}\label{rem:tedious_calculation}
Let $\zeta$ be any primitive $d$th root of unity, for $d$ dividing $2n-k$. 
Then \[ f(\zeta) = \begin{cases} \displaystyle \frac{(\frac{2n-k}{d})!}{(\frac{k}{d})! (\frac{n-k}{d})!(\frac{n-k}{d}-1)!\frac{n}{d}}, & \mathrm{if} \, d | n \\[20pt] 
\displaystyle \frac{(\frac{2n-k}{d})!}{(\frac{k+2}{d} - 1)!(\frac{n-k-1}{d})!(\frac{n-k-1}{d})!\frac{n+1}{d}}, & \mathrm{if} \, d | n+1 \\[20pt]
0, & \mathrm{otherwise}. \end{cases} \]
\end{lemma}
\begin{proof}
As in \cite[$\mathsection 7$]{reiner.stanton.white}, we observe that 
\begin{itemize}
\item for all $j \in \N$, $\zeta$ is a root of $[j]_q$ if and only if $d > 1$ divides $j$, and
\item for all $j, j' \in \N$ with $j \equiv j' \mod d$, 
\[\lim_{q \to \zeta} \frac{[j]_q}{[j']_q} = \begin{cases} \displaystyle \frac{j}{j'}, & \mathrm{if} \, j \equiv 0 \mod d \\[18pt]
1, & \mathrm{if} \, j \not \equiv 0 \mod d. \end{cases}
\]
\end{itemize}
The desired formula is then an easy but tedious calculation analogous to those carried out in \cite[$\mathsection 7$]{reiner.stanton.white}.
\end{proof}

We will write $\pi$ for the bijection of Proposition \ref{prop:partition_bijection} from $\twinc{n}{k}$ to noncrossing partitions of $2n-k$ into $n-k$ blocks all of size at least 2. For $\Pi$ a noncrossing partition of $N$, we write $\mathcal{R}(\Pi)$ for the noncrossing partition given by rotating $\Pi$ clockwise by $2 \pi / N$.

\begin{lemma}\label{lem:promotion_is_rotation_of_partitions}
For any $T \in \twinc{n}{k}$, $\pi(\pro(T)) = \mathcal{R}(\pi(T))$.
\end{lemma}
\begin{proof}
We think of K-promotion as taking place in two steps. In the first step, we remove the label 1 from the tableau $T$, perform K-jeu de taquin, and label the now vacated lower-right corner by $2n-k+1$. Call this intermediate filling $T'$. In the second step, we decrement each entry by one to obtain $\pro(T)$. 

The filling $T'$ is not strictly an increasing tableau, as no box is labeled 1. However, by analogy with the construction of Proposition \ref{prop:partition_bijection}, we may associate to $T'$ a noncrossing partition of $\{2, \dots, 2n-k+1\}$. For each $i$ in the second row of $T'$, let $s_i$ be the largest number in the first row that is less than $i$ and that is not $s_j$ for some $j<i$. The noncrossing partition $\pi(T')$ is formed by declaring $i$ and $s_i$ to be in the same block. 

\begin{claim}$\pi(T')$ may be obtained from $\pi(T)$ merely by renaming the element 1 as $2n-k+1$. 
\end{claim}

In $T$, there are three types of elements: \begin{enumerate}[(A)]
\item those that appear only in the first row, 
\item those that appear only in the second, and 
\item those that appear in both. \end{enumerate} Most elements of $T'$ are of the same type as they were in $T$. In fact, the only elements that change type are in the block of $\pi(T)$ containing 1. If that block contains only one element other than 1, this element changes from type (B) to type (A). If the block contains several elements besides 1, the least of these changes from type (C) to type (A) and the greatest changes from type (B) to type (C). All other elements remain the same type. Observing that the element 1 was of type (A) in $T$ and that $2n-k+1$ is of type (B) in $T'$, this proves the claim. \qed

By definition, $\pi(\pro(T))$ is obtained from $\pi(T')$ merely by decrementing each element by one. However by the claim, $\mathcal{R}(\pi(T))$ is also obtained from $\pi(T')$ by decrementing each element by one. Thus $\pi(\pro(T))  = \mathcal{R}(\pi(T))$.
\end{proof}

It remains now to count noncrossing partitions of $2n-k$ into $n-k$ blocks all of size at least 2 that are invariant under rotation by $2 \pi /d$, and to show that we obtain the formula of Lemma \ref{rem:tedious_calculation}. We observe that for such rotationally symmetric noncrossing partitions, the cyclic group $\mathcal{C}_d$ acts freely on all blocks, except the central block (the necessarily unique block whose convex hull contains the center of the circle)  if it exists. Hence there are no such invariant partitions unless $n-k \equiv 0$ or $1 \mod d$, in agreement with Lemma \ref{rem:tedious_calculation}.

Arrange the numbers $1,2, \dots, n, -1, \dots, -n$  counterclockwise, equally-spaced around a circle. Consider a partition of these points such that, for every block $B$, the set formed by negating all elements of $B$ is also a block. If the convex hulls of the blocks are pairwise nonintersecting, we call such a partition a \emph{noncrossing $B_n$-partition} or \emph{type-B noncrossing partition} (\emph{cf.}~\cite{reiner}). Whenever we say that a type-B noncrossing partition has $p$ pairs of blocks, we do not count the central block. There is an obvious bijection between noncrossing partitions of $2n-k$ that are invariant under rotation by $2 \pi /d$ and noncrossing $B_{(2n-k)/d}$-partitions. Under this bijection a noncrossing partition $\Pi$ with $n-k$ blocks corresponds to a type-B noncrossing partition with $\frac{n-k}{d}$ pairs of blocks if $d$ divides $n-k$ (that is, if $\Pi$ has no central block), and corresponds to a type-B noncrossing partition with $\frac{n-k-1}{d}$ pairs of blocks if $d$ divides $n-k-1$ (that is, if $\Pi$ has a central block). The partition $\Pi$ has singleton blocks if and only if the corresponding  type-B partition does. 

\begin{lemma}\label{lem:enumeration_of_type_B_partitions}
The number of noncrossing $B_N$-partitions with $p$ pairs of blocks without singletons and \emph{without} a central block is $$\sum_{i=0}^p (-1)^i \binom{N}{i} {\binom{N-i}{p-i}}^2 \frac{p-i}{N-i}.$$ The number of such partitions \emph{with} a central block is $$\sum_{i=0}^p (-1)^i \binom{N}{i} {\binom{N-i}{p-i}}^2 \frac{N-p}{N-i}.$$
\end{lemma}
\begin{proof}
It was shown in \cite{reiner} that the number of noncrossing $B_N$-partitions with $p$ pairs of blocks is ${\binom{N}{p}}^2$.  In \cite{reiner.stanton.white}, it was observed that \cite[Lemma 4.4]{athanasiadis.reiner} implies that exactly $\frac{N-p}{N} {\binom{N}{p}}^2$ of these have a central block. Our formulas for partitions without singleton blocks follow immediately from these observations by Inclusion--Exclusion.
\end{proof}

It remains to prove the following pair of combinatorial identities:
$$\sum_{i=0}^p (-1)^i \binom{N}{i} {\binom{N-i}{p-i}}^2 \frac{p-i}{N-i} = \frac{N!}{(N-2p)!p!(p-1)!(N-p)}$$
and
$$\sum_{i=0}^p (-1)^i \binom{N}{i} {\binom{N-i}{p-i}}^2 \frac{N-p}{N-i} = \frac{N!}{(N-2p-1)!p!p!(N-p)}.$$

This is a straightforward exercise in hypergeometric series (\emph{cf.},~e.g.,~\cite[$\mathsection 2.7$]{andrews.askey.roy}). For example, the first sum is 
\[\binom{N}{p} \sum_{i=0}^p (-1)^i \binom{p}{i} \binom{N-i-1}{p-i-1} = \binom{N}{p} \binom{N-1}{p-1} \pFq{2}{1}{-p, 1-p}{1-N}{1},\] which may be evaluated by the Chu--Vandermonde identity. This completes the proof of Theorem \ref{thm:promotion_CSP}. \qed


Recently, C.~Athanasiadis--C.~Savvidou \cite[Proposition 3.2]{athanasiadis.savvidou} independently enumerated noncrossing $B_N$-partitions with $p$ pairs of blocks without singletons and without a central block.

Lemma \ref{lem:promotion_is_rotation_of_partitions} yields a second proof of Theorem \ref{thm:orbit_size}. We observe that under the reformulation of Lemma \ref{lem:promotion_is_rotation_of_partitions}, Theorem \ref{thm:promotion_CSP} bears a striking similarity to \cite[Theorem 7.2]{reiner.stanton.white} which gives a cyclic sieving on the set of \emph{all} noncrossing partitions of $2n-k$ into $n-k$ parts with respect to the same cyclic group action.

 Additionally, under the correspondence mentioned in Section \ref{sec:enumerations} between $\twinc{n}{k}$ and dissections of an $(n+2)$-gon with $n-k -1$ diagonals, Theorem \ref{thm:promotion_CSP} bears a strong resemblance to \cite[Theorem 7.1]{reiner.stanton.white}, which gives a cyclic sieving on the same set with the same $q$-enumerator, but with respect to an action by $\mathcal{C}_{n+2}$ instead of $\mathcal{C}_{2n-k}$. S.-P.~Eu--T.-S.~Fu \cite{eu.fu} reinterpret the $\mathcal{C}_{n+2}$-action as the action of a Coxeter element on the $k$-faces of an associahedron. We do not know such an interpretation of our action by $\mathcal{C}_{2n-k}$. In \cite{reiner.stanton.white}, the authors note many similarities between their Theorems 7.1 and 7.2 and ask for a unified proof. It would be very satisfying if such a proof could also account for our Theorem \ref{thm:promotion_CSP}.

\section*{Acknowledgements}
The author was supported by an Illinois Distinguished Fellowship from the University of Illinois, an NSF Graduate Research Fellowship, and NSF grant DMS 0838434 ``EMSW21MCTP: Research Experience for Graduate Students.''

Aisha Arroyo provided help with some early computer experiments. The author thanks Victor Reiner for helpful comments on an early draft and Christos Athanasiadis for bringing \cite{athanasiadis.savvidou} to his attention. The author is grateful for helpful conversations with Philippe Di Francesco, Sergey Fomin, Rinat Kedem, Amita Malik, Brendon Rhoades, Luis Serrano, Jessica Striker, and Hugh Thomas.  Brendon Rhoades also provided critical assistance computing with Kazhdan--Lusztig representations. The author is especially grateful to Alexander Yong for his many useful suggestions about both the mathematics and the writing of this paper. Two anonymous referrees provided very thoughtful comments and suggestions.

\bibliographystyle{alpha}
\hspace*{1cm}
\bibliography{JCTA}

\end{document}